\documentclass[12pt]{article}
\usepackage[cp1251]{inputenc}
\usepackage[T2A]{fontenc}
\usepackage[english, russian]{babel}
\usepackage{amssymb}
\usepackage{amsfonts}
\usepackage{amsmath}
\usepackage{amscd}
\usepackage[english,russian]{babel}
\usepackage{wrapfig,epsfig}
\usepackage{mathrsfs}
\usepackage{newlfont}

\title{On uniformly consistent nonparametric tests}

\author{Mikhail Ermakov }
\date{April 2023}

\numberwithin{equation}{section}

\newtheorem{thm}{Theorem}[section]

\newtheorem{corollary}{Corollary}[section]

\begin{document}
\global\long\def\zb{\boldsymbol{z}}
\global\long\def\ub{\boldsymbol{u}}
\global\long\def\vb{\boldsymbol{v}}
\global\long\def\wb{\boldsymbol{w}}
\global\long\def\tb{\boldsymbol{t}}
\global\long\def\eb{\boldsymbol{e}}
\global\long\def\sb{\boldsymbol{s}}
\global\long\def\yb{\boldsymbol{y}}
\global\long\def\thetab{\boldsymbol{\theta}}
\global\long\def\0b{\boldsymbol{0}}
\global\long\def\xib{\boldsymbol{\xi}}
\global\long\def\taub{\boldsymbol{\tau}}
\global\long\def\etab{\boldsymbol{\eta}}
\global\long\def\zetab{\boldsymbol{\zeta}}
\maketitle
Institute of Problems of Mechanical Engineering RAS, Bolshoy pr., 61, VO, 1991178  St. Petersburg and
St. Petersburg State University, Universitetsky pr., 28, Petrodvoretz, 198504 St. Petersburg, RUSSIA

AMS subject classification:
62F03, 62G10, 62G2

keywords :
consistency, uniform consistency,
goodness of fit tests, signal detection.

\maketitle
\begin{abstract} Necessary and sufficient conditions for existence of uniformly consistent tests are explored. A hypothesis is simple. Nonparametric sets of alternatives are bounded convex sets in $\mathbb{L}_p$, $p \ge 2$ with "small" balls   deleted. The "small" balls have the center at the point of hypothesis and radii of balls tend to zero as sample size increases. For problem of hypothesis testing on a density, we show that, for the sets of alternatives, there are  uniformly consistent tests for some sequence of radii of the balls, if and only if, convex set is relatively compact. Similar results are established for problem of signal detection in Gaussian white noise, for linear ill-posed problems  with random Gaussian noise and so on.
\end{abstract}

\section{Introduction}

Compactness assumption is necessary condition for existence of uniformly consistent estimator \cite{ih}. This assumption is also naturally arises if we explore ill-posed problems with deterministic noise \cite{en}. In some sense compactness is rather severe assumption allowing to find uniformly consistent estimator. Now we shall see that the same severe assumption is necessary in  natural setups of nonparametric hypothesis testing.

We  say that family of sets of alternatives are uniformly consistent if, for these sets of alternatives, there is uniformly consistent  family of tests.
Uniform consistency of nonparametric sets  of alternatives has been explored from different viewpoints and different setups. A rather complete bibliography one can find in \cite{er15}.
We explore the problem, if hypothesis is simple and nonparametric sets of alternatives are approaching with hypothesis. The research are based on Le Cam -Schwartz \cite{les}  results.
 Le Cam and Schwartz \cite{les} have found necessary and sufficient conditions for existence of uniformly consistent tests and estimators. These conditions are rather cumbersome and are provided in terms of $n$-fold products of probability measures. A simple form of this conditions arises if probability measures of statistical experiment  are  compact in the topology of set-wise weak convergence, other words, if densities of probability measures are uniformly integrable.

 By Lemma 2.3 in Gaensler \cite{ga}, if set is compact in weak topology of set-wise convergence and there is  weaker topology on this set such that the set is Hausdorff space for this topology then these two topologies coincide on the set. This allows to point out necessary and sufficient conditions of uniform consistency in terms of simple natural weak topologies for widespread statistical models: hypothesis testing on a density, signal detection in Gaussian white noise, hypothesis testing on intensity of Poisson process, hypothesis testing on Gaussian covariance operator and so on (see \cite{er15}). We can consider this research as natural complement to previous interesting  numerous results on consistency of nonparametric tests \cite{bir,dep,er21,ho,ih,is,ja,kr,pf}.

 In \cite{er15}, sets of alternatives did not depend on sample size (power of signal and so on). For a long time the necessary and sufficient conditions of uniform consistency for approaching sets of alternatives remained unclear. Special examples of indistinquishability of hypotheses have been proposed. The first result on indistinguishability  of hypotheses and alternatives has been established Le Cam \cite{le73}. For  hypothesis testing
on a density Le Cam \cite{le73} has shown that the center of the ball and the interior of the ball in $L_1$ are indistinguishable. For the problem of signal detection in Gaussian white noise Ibragimov and Hasminski \cite{ih} have proved  that the center of a ball and the interior of a ball in $L_2$ are indistinguishable. For the same problem Burnashev \cite{bu} established the indistinguishability of
the interior of a ball in $L_p$- spaces, $p>1 $.
  For other setups we mention Ingster \cite{in93}
and Ingster, Kutoyants \cite{ik} works.  Ingster \cite{in93}
and Ingster and Kutoyants \cite{ik} showed  that we could not distinguish the center of a ball and the interior of a ball in $L_2$  in the problems of  hypothesis
testing on a density  and on an intensity function of Poisson process  respectively.

 In  \cite{er21}, we showed that, if sets of alternatives  are bounded center-symmetric convex set with deleted "small" balls in
 $\mathbb{L}_2$ having center at the zero point of hypothesis, then  there are uniformly consistent tests for some sequence of radii of balls tending to zero as sample size increases, if and  only if, the center-symmetric convex set is relatively compact.

 Paper goal is to prove similar results without assumption of center-symmetry of convex set and to extend this result for deleted balls in $\mathbb{L}_p$, $p\ge 2$. The technique allows to explorer  balls from another Banach spaces. However we does not touch these setups.

   \section{Uniformly consistent sets}

 \subsection{ Hypothesis testing on a density}

 Let $X_1,\ldots,X_n$ be i.i.d.r.v.'s having probability measure $\mathbf{P}$   defined on Borel sets of $\sigma$-field $\frak{B}$ assigned on a set $S$. Denote $\cal{P}$ the set of all probability measures on $(S,\frak{B})$.

 For probability measure $\mathbf{P}_0 \in \cal{P}$ denote $\mathbb{L}_p$, $2 \le p < \infty$, the set of all real measurable functions $g$ such that
 \begin{equation*}
 \|g\|_p^p = \int_S |g|^p d \, \mathbf{P}_0 < \infty.
 \end{equation*}
 For each $\rho >0$ and for each $u \in \mathbb{L}_p$, define the ball $B_\rho(u) = \{\,g\,:\, \|g - u\|_p \le \rho,\, g \in \mathbb{L}_p\}$.

 Fix bounded convex set $U \subset \mathbb{L}_p$.

 We say that $u \in U$ is internal  point of convex set $U$ if, for any  function $v \in U$, there is $\lambda < 0$ such that $u + \lambda (v - u) \in U$. Note that traditional topological definition of internal point of set  is different (p.410 \cite{dun}).

 Suppose that $f_0 \equiv 0$ is internal point of $U$.

 We explore problem of testing hypothesis
 \begin{equation*}
 \mathbb{H}_0\,\,:\,\,f \doteq \frac{d \mathbb{P}}{d \mathbb{P}_0} - 1  = f_0 \doteq 0
 \end{equation*}
 versus alternatives
 \begin{equation*}
 \mathbb{H}_n\,\,:\,\, f \in V(\rho_n) = U \setminus B_{\rho_n}(f_0)
 \end{equation*}
 with $\rho_n > 0$, $\rho_n \to 0$ as $n \to \infty$.

 For any test $K_n$ denote $\alpha(K_n) = \mathbf{E}_0[K_n]$ its type I error probability and denote $\beta(K_n,f)$ its type II error probability for alternative $f \in V(\rho_n)$.

 We say that sequence of tests $K_n$, $ 0 < \alpha(K_n,f_0) \le \alpha <1$, is uniformly consistent, if there is $n_0 >0$  and $\delta > 0$ such that, for all $n > n_0$, we have
 \begin{equation*}
 \beta(K_n,V(\rho_n)) \doteq \sup_{f  \in V(\rho_n)} \beta(K_n,f) < 1 - \alpha- \delta
 \end{equation*}
 Otherwise, we say that sequence of tests $K_n$, $ 0 < \alpha(K_n) \le \alpha <1$, is uniformly inconsistent.
 \begin{thm} \label{th1} There is a sequence $\rho_n \to 0$ as $n \to \infty$ such that  sets of alternatives $V(\rho_n)$ is uniformly consistent, if and only if, set $U$ are relatively compact in $\mathbb{L}_p$.
 \end{thm}
 We say that $r_n$ is rates of distinguishability if sets of alternatives $V(\rho_{1n})$ is uniformly inconsistent for any sequence $\rho_{1n} > 0$, $\rho_{1n}/r_n \to 0$ as $n \to \infty$, and
 sets of alternatives $V(\rho_{2n})$ is uniformly consistent for any sequence $\rho_{2n} > 0$, $\rho_{2n}/r_n \to \infty$ as $n \to \infty$.

 In proof of Theorem \ref{th1} we construct set $W$ such  that $U \subset W$ and show that Kolmogorov $n$-width $d_i$, $1\le i < \infty$ of set $W$ tends to zero. The same reasoning we can conduct for any internal point $u \in U$ and construct set $W_u$,  $U \subset W_u$ with Kolmogorov $n$-width $d_{ui}$, $1\le i < \infty$. If we take $h = u$ in proof of Theorem \ref{th1} and, by the same reasoning, implementing similarity inequality (\ref{pr21}), we get $c d_i \le d_{ui} \le C d_i$, $1 \le i < \infty$, where coefficients $c, C$ are proportional to  similarity coefficient.

For $\varepsilon > 0$, $c>1$ and $\theta \in K$, denote $M(\varepsilon, c, B_\varepsilon(\theta)\cap K)$,  the largest cardinality of an $\varepsilon/c$ -packing set for $B_\varepsilon(\theta)\cap K$ (see definition 5.4 in \cite{wainb}). We refer to
\begin{equation*}
\frak{E}^{loc}(\theta,c,\varepsilon) =\log M(\varepsilon,c,B_\varepsilon(\theta)\cap K)
 \end{equation*}
as local entropy of set $K$ at the point $\theta$. The same similarity consideration allow to show that  there are $c_1$ and $c_2$ such that, for any interior points $\theta_1 \in K$ and $\theta_2 \in K$ we have
\begin{equation*}
\frak{E}^{loc}(\theta_1,c_2/c_1,c_1\varepsilon) \le
\frak{E}^{loc}(\theta_2,c_2,\varepsilon) \le \frak{E}^{loc}(\theta_2,c_2/c_1, c_1\varepsilon).
 \end{equation*}

 This is the base for the proof that we have the same rates of distinguishability in the problems of hypothesis testing $\mathbb{H}_0\,:\, f = 0$ versus $\mathbb{H}_n \,:\, f \in U \setminus B_{\rho_n}(0)$ and $\mathbb{H}_{0u}\,:\, f = u$ versus $\mathbb{H}_{nu} \,:\, f \in U \setminus B_{\rho_n}(u)$ if $u$ is internal point of $U$.  If  point $u$ depends on $n$, one can get different rates of distinguishability  \cite{wain}.

    \begin{corollary} Let a priori information be provided that $f \in U$ where $U$ is bounded convex set in $\mathbb{L}_p$, $p \ge 2$. Then there is uniformly consistent estimator of  $\|f\|_p$, if and only if,  the set $U$ is relatively compact.
   \end{corollary}

        Proof of Theorem \ref{th1} is based on  Theorem \ref{th2} on uniform consistency of set of alternatives that does not depend on $n$.

        For $p > 1$, we explore the problem of testing hypothesis
\begin{equation*}
\mathbb{H}_0 \,\, :\,\, f \in W_0 \subset \mathbb{L}_p
\end{equation*}
versus alternatives
\begin{equation*}
\mathbb{H}_1\,\,:\,\, f \in W_1 \subset \mathbb{L}_p
\end{equation*}
We say that sequence of tests $K_n$ is uniformly consistent if
\begin{equation*}
\lim\sup_{n \to \infty}\,(\sup_{f_0 \in W_0} \alpha(K_n,f_0) + \sup_{f \in W_1} \beta(K_n,f)) < 1
\end{equation*}
\begin{thm}\label{th2} {\sl i.} Let $W_0$ be bounded set in $L_p$, $p>1$. Then there are uniformly consistent tests  if  the closures of $W_0$ and $W_1$ are disjoint in weak topology.

{\sl ii}. The converse holds if $W_1$ is also bounded sets in $L_p$, $p>1$.
\end{thm}
Theorem \ref{th2} follows from Le Cam, Schwartz (p.141, \cite{les})  and Lemma 2.3 in Ganssler \cite{ga} (see Theorem 4.1 in \cite{er15} as well).

If closures in weak topology of sets $W_0$ and $W_1$ are disjoint then there is a finite number of functions $g_1, \ldots, g_k \in \mathbb{L}_q$, $ \frac{1}{p} + \frac{1}{q} =1$, and constants $c_1, \ldots,c_k$ such that, for any $u \in W_0$ there holds $c_i\, >\, \int_S\, u\, g_i\, d\, \mathbf{P}_0$ for all $i$, $1 \le i \le k$, and, for any $v \in W_1$, there is index $i$, $1 \le i \le k$, such that
$\int_S\, v\, g_i\, d\, \mathbf{P}_0\, >\,c_i$.
\subsection{Hypothesis testing on intensity function of Poisson random process \label{s2.3}} Let we be given $n$ independent realizations $\kappa_1,\ldots,\kappa_n$ of Poisson random process with mean measure $\mathbf{P}$  defined on Borel sets of $\sigma$-field $\frak{B}$ assigned on set $S$. Suppose mean measure $\mathbf{P}$ has intensity function $h(t) = \frac{d \mathbf{P}}{d \nu}(t)$, $t \in S$. Here $\nu$ is a measure  defined on Borel sets of $\sigma$-field $\frak{B}$ as well and $\nu(S) < \infty$, $\mathbf{P}(S) < \infty$.

Problem is to test a hypothesis $H_0\,:\, f(t)\doteq f_0 = h(t) -1 \equiv 0$ versus  alternatives $H_n\,:\, f \in V(\rho_n) \subset \mathbb{L}_p$, $p \ge 2$. The definition of sets $V(\rho_n)$, $\rho_n > 0$, is the same as in previous subsection with the unique difference that the densities are replaced with intensity functions. The definition of internal point of convex set $U$ is also the same as in the previous subsection.
\begin{thm} \label{thp1}  Let $ U$ be bounded convex set in $\mathbb{L}_p$, $p \ge 2$. Let $f_0$ be internal point of $U$. Then there is sequence $\rho_n \to 0$ as $n \to \infty$ such that  sets of alternatives $V(\rho_n)$ are uniformly consistent, if and only if,  set $U$ is relatively compact.
 \end{thm}
Proof of Theorem \ref{thp1} completely coincides with the proof of Theorem \ref{th1}. It suffices only to implement  Theorem 5.9 in \cite{er15} instead of Theorem \ref{th2}.
\subsection{Signal detection  in Gaussian white noise \label{s2.2}}

 Suppose we  observe a realization of stochastic process $Y_\epsilon(t), t \in (0,1)$, defined by the stochastic differential equation
$$
dY_\varepsilon(t) = S(t) dt + \varepsilon dw(t), \quad \varepsilon > 0,
$$
where $S\in \mathbb{L}_2(0,1)$ is unknown signal and $dw(t)$ is Gaussian white noise.

Let $U $ be bounded convex set in $\mathbb{L}_2(0,1)$. We say that $S_0 \in U$ is internal point of convex set $U$ if, for any $S \in U$, there is $\lambda <0$ such that $S_0 +  \lambda (S - S_0) \in U$.

Denote $B_r(S)$ ball in $\mathbb{L}_p$, $p \ge 2$, having radius $r>0$ and center $S \in \mathbb{L}_2$.

 We wish to test a hypothesis
 \begin{equation*}
 \mathbb{H}_0\,\,:\,\, S(t)  = S_0(t) \equiv  0, \quad t \in [0,1],
 \end{equation*}
 versus alternatives
 \begin{equation*}
 \mathbb{H}_\varepsilon\,\,:\,\, S \in V(\rho_\varepsilon) = U \setminus B_{\rho_\varepsilon}(S_0)\quad \rho_\varepsilon > 0.
 \end{equation*}
 For any test $K_\varepsilon$, $\varepsilon > 0$, denote $\alpha(K_\varepsilon) = \mathbf{E}_0[K_\varepsilon]$ its type I error probability and denote $\beta(K_\varepsilon,S)$ its type II error probability for alternative $S \in V(\rho_\varepsilon)$.

 We say that family of tests $K_\varepsilon$, $ 0 < \alpha(K_\varepsilon) \le \alpha <1$, $\varepsilon > 0$, is uniformly consistent if there is $\varepsilon_0 >0$  and $\delta > 0$ such that, for all $0 < \varepsilon < \varepsilon_0$, we have
 \begin{equation*}
 \beta(K_\varepsilon,V(\rho_\varepsilon)) = \sup_{S  \in V(\rho_\varepsilon)} \beta(K_\varepsilon,S) < 1 - \alpha- \delta.
 \end{equation*}
\begin{thm} \label{ths1}  Let $ U$ be bounded convex set in $\mathbb{L}_p$, $p \ge 2$. Let $S_0$ be internal point of $U$. Then there is  $\rho_\varepsilon \to 0$ as $\varepsilon \to 0$ such that  the sets of alternatives $V(\rho_\varepsilon)$ are uniformly consistent, if and only if, the set $U$ is relatively compact.
 \end{thm}
Proof of Theorem \ref{thp1} completely coincides with proof of Theorem \ref{th1}. It suffices to implement  only Theorem 5.3 in \cite{er15} instead of Theorem \ref{th2}.
\subsection{Hypothesis testing on a solution of linear ill-posed problem \label{s2.4}}
In Hilbert space $\mathbb{H}$
  we observe a realization of   Gaussian random vector
 \begin{equation}\label{il1}
  Y_\varepsilon = A S + \varepsilon \xi,  \quad S \in \mathbb{H}, \quad \varepsilon > 0.
  \end{equation}
  Hereafter $A: \mathbb{H} \to \mathbb{H}$ is known  linear operator and $\xi$ is Gaussian random vector having known covariance operator $R: \mathbb{H} \to \mathbb{H}$ and
  $\mathbf{E}\xi = 0$.

  Denote $B_r(S)$ ball in $\mathbb{H}$,  having radius $r>0$ and center $S \in \mathbb{H}$.

  Let $U $ be bounded convex set in $\mathbb{H}$.

We wish to test a hypothesis
 \begin{equation*}
 \mathbb{H}_0\,\,:\,\, S  = S_0 =  0,
 \end{equation*}
 versus alternatives
 \begin{equation*}
 \mathbb{H}_\varepsilon\,\,:\,\, S \in V_(\rho_\varepsilon) = U \setminus B_{\rho_\varepsilon}(S_0), \quad \rho_\varepsilon > 0.
 \end{equation*}
 Below  we use the same definition of uniform consistency and similar definition of interior point as in subsection \ref{s2.2}.

For any operator $K: \mathbb{H} \to \mathbb{H}$ denote $\frak{R}(K)$ the rangespace of $K$.

  Suppose that the nullspaces of $A$ and $R$ equal  zero and $\frak{R}(A) \subseteq \frak{R}(R^{1/2})$.

\begin{thm}\label{t21} Let the operator $R^{-1/2}A$ be bounded in $\mathbb{H}$. Let $S_0$ be internal point of $U$. Then there is  $\rho_\varepsilon \to 0$ as $\varepsilon \to 0$ such that  sets of alternatives $V(\rho_\varepsilon)$ are uniformly consistent, if and only if, the set $U$ is relatively compact.\end{thm}

Proof of Theorem \ref{thp1} completely coincides with proof of Theorem \ref{th1}. It suffices to implement  only Theorem 5.4 in \cite{er15} instead of Theorem \ref{th2}

\subsection{Hypothesis testing on an operator known with a random error}
Let $\Theta$ be the set of all bounded linear operators $\mathbb{H }\to \mathbb{H}$ where $\mathbb{H}$ is  separable Hilbert space.

Let $A \in \Theta$ be an unknown linear operator.

Let for all $S \in H$ we can observe random Gaussian vector
\begin{equation*}
Y_{7\varepsilon,S} = A\,S + \epsilon\xi_S
\end{equation*}
with $\mathbf{E} \xi_S = 0$ and $\mathbf{E}<\xi_{S_1},\xi_{S_2}> = <S_1,S_2>$ where $<S_1,S_2>$ is inner product of vectors $S_1,S_2 \in \mathbb{H}$.

Let $U$ be convex set in $\Theta$.

We say that $A_1\in  U$ is internal point of $U$ if, for any $A \in U$, there is $\lambda < 0$ such that $A_1 + \lambda (A- A_1) \in U$.

Denote $A_0 \in \Theta$ linear operator such that, for any $S \in \mathbb{H}$, we have $A_0 S = 0$.

For Hilbert-Schmidt norm in $\Theta$, denote $B(r)$ ball in $\Theta$ having radius $r> 0$ and center $A_0$.

Let $A_0$ be interior point of $U$.

The problem is to test  hypothesis $H_0\,:\, A = A_0$ versus alternatives $H_1\,:\, A \in V(\rho_\varepsilon)\doteq U\setminus B(\rho_\varepsilon)$ with $\rho_\varepsilon > 0$.

The definition of uniform consistency is akin to one in subsection \ref{s2.2}.

 In the following Theorem we consider the boundedness and compactness for Hilbert-Schmidt norm.

 \begin{thm}\label{tho1} Let $U$ be bounded convex set and let $A_0$ be internal point of $U$. Then there is  $\rho_\varepsilon \to 0$ as $\varepsilon \to 0$ such that  sets of alternatives $V(\rho_\varepsilon)$ are uniformly consistent, if and only if, the set $U$ is relatively compact.
 \end{thm}
The reasoning are akin to proof of Theorem 2.1. It suffices only to implement Theorem 5.7 in \cite{er15}.
\section{Proof of Theorem \ref{th1}}

 Suppose that there is uniformly consistent sequence of alternatives $V(\rho_n)$, $\rho_n \to 0$ as $n \to \infty$. Then we implement Theorem \ref{th2} and point out compact set $W$ such that $ U \subset W$.

Fix $\varepsilon_1 > 0$. For each $ i > 1$, we put $\varepsilon_i = \frac{\varepsilon_{i-1}}{2}$.

By Theorem \ref{th2}, for each $i$, there are functions $g_{i1},\ldots,g_{ik_i} \in \mathbb{L}_q$, $\|g_{ij}\|_q =1$ for $1 \le j \le k_i$, $\frac{1}{p} + \frac{1}{q} =1$, and constants $c_{i1},\ldots,c_{ik_i}$ such that
\begin{equation}\label{pr1}
V(\varepsilon_i) \subset \cup_{l=1}^i\cup_{j=1}^{k_l} \Pi(g_{lj},c_{lj})
\end{equation}
and
\begin{equation}\label{pr2}
0 \notin \cup_{l=1}^i\cup_{j=1}^{k_l} \Pi(g_{lj},c_{lj}),
\end{equation}
where
\begin{equation*}
 \Pi(g_{lj},c_{lj}) = \Bigl\{ \, f \,:\, \int_S g_{lj}\, f\, d \, \mathbf{P}_0  > c_{lj}, f \in \mathbb{L}_p\Bigr\}.
\end{equation*}
Denote $\Gamma(g_{lj},c_{lj})$ the boundary hyperplane of  $\Pi(g_{lj},c_{lj})$.

Define sets $G_i = \{g_{i1},\ldots,g_{ik_i}\}$, $ 1 \le i < \infty$.

It is clear, that we can suppose additionally that, if $g \in G_i$, then $-g \in G_i$, $ 1\le i < \infty$. Otherwise we can add such functions. Moreover, we can suppose that, if $g_{ij_1} = - g_{ij_2}$, $1 \le j_1 \ne j_2 \le k_i$, then $ c_{ij_1} = c_{ij_2}$.

We can also suppose that $g_{ij}$ are step functions and therefore $g_{ij} \in \mathbb{L}_p\cap \mathbb{L}_q$, $1 \le j  \le k_i$, $1 \le i < \infty$. Otherwise we can replace functions functions $g_{ij}$ by their step approximations and change constants $c_{ij}$ insignificantly. We can also choose functions $g_{ij}$ such that function $g_{i_1j_1}$ is orthogonal to function $g_{i_2,j_2}$ in  $\mathbb{L}_2$, for any $ 1 \le i_1 \ne i_2 < \infty$ and for any $1 \le j_1 < k_{i_1}$, $1 \le j_2 < k_{i_2}$. Suppose also that $\|g_{ij}\|_q = 1$  for $1 \le j \le k_i$, $ 1\le i < \infty$.

For $i$, $1 \le i  < \infty$, denote $M_i$ linear subspace generated functions $g_{i1}, \ldots, g_{ik_i}$.

Denote $B(M_i,\rho)$ ball in subspace $M_i$ of $\mathbb{L}_p$ having center zero and radius $\rho > 0$.

Denote $\Omega_i = M_i \cap \cap_{j=1}^{k_i}\bar\Pi (g_{ij},\varepsilon_{i-1})$ where
$\bar\Pi(g_{ij},d_{1j}) = \mathbb{L}_p\setminus \Pi(g_{ij},d_{1})$.

Fix $\delta > 0$.   We can take the number of functions $g_{ij}$, $\|g_{ij}\|_q =1$, $1 \le j \le k_i$, so large that, for any $f \in \Omega_i$, we have
\begin{equation}\label{toto}
\|f\|_p < \frac{\varepsilon_{i-1}}{1- \delta}.
\end{equation}
Denote $d_{i} \doteq \frac{\varepsilon_{i-1}}{1- \delta}$.

For definition of set $W$ we implement induction by $i$.

Since  set $U$ is bounded there is $d_{1}$ such that
for any $j$, $1 \le j \le k_1$, and, for any $f \in U$, we have $f \in \bar\Pi(g_{1j},d_{1})$.

Let $v \in U \setminus B_{\varepsilon_1}(0)$, $\int_S g_{11} v d\,\mathbf{P}_0 > c_{11}$.
Then there is $\lambda < 0$ such that $h \doteq \lambda v \in U$ and $1 + \lambda v$ is a density.

Denote $\Pi_0 = \Bigl\{\,f\,:\, \int_S g_{11} f d\,\mathbf{P}_0 > 0, f \in \mathbb{L}_p\Bigr\}$.
Let $\Gamma_0$ be the boundary hyperplane of $\Pi_0$.

Denote $W_1 = \cap_{j=1}^{k_1} \bar\Pi (g_{1j},d_{1})$.

Let $i > 1$. For each $j$, $1 \le j\le k_i$, we draw  hyperplane $\Upsilon_{ij}$ passing through the point $h$ and $\Gamma(g_{ij},d_{i})\cap \Gamma_0$. Denote $\Psi_{i,j}$ the set having hyperplane  $\Upsilon_{i,j}$ as a boundary and containing $(U\setminus B(\varepsilon_{i-1})) \cap \Pi_0$.

Denote $W_i = \cap_{j=1}^{k_i} \Psi_{ij} \cap W_{i-1} $.

Then $U \cap \Pi_0\subset W_i$.

Denote $L_i$ linear space generated functions from linear spaces $M_1,\ldots,M_i$.

Let $R>0$ be such that $U \subset B_R(0)$.

Using triangle similarity considerations and (\ref{toto}), we have
\begin{equation} \label{pr21}
\begin{split} &
\sup_{f}\,\,\,\inf_{f_1 }\,\, \{\|f - f_1\|_p\,:\,f \in \Upsilon_{i,j}\cap W_i\cap L_i,\,\, f_1 \in W_{i-1}\cap L_{i-1} \,\} \\&
\le
 |\lambda|^{-1} R\,\sup_{f} \{\|f - f_1\|_p\,:\,f \in \Upsilon_{i,j}\cap W_i\cap L_i \cap \Gamma_0, f_1 \in W_{i-1}\cap L_{i-1}\cap\Gamma_0\} \\&
\le d_{i} |\lambda|^{-1} R .
\end{split}
\end{equation}
for any $i$, $1 < i < \infty$, and any $j$, $1 \le j \le k_i$.

Denote $W$ the closure of $\cap_{i=1}^\infty W_i$. By (\ref{pr21}), $n$-width of W tends to zero, therefore $W$ is relatively compact. Since $U\cap \Pi_0 \subset W$ then $U\cap \Pi_0 $ is  relatively compact as well. The case $U\setminus\Pi_0$ is explored similarly.

\end{document}